\documentclass[10pt, letterpaper]{article}

\usepackage{amsmath, amssymb, amsthm, graphicx, amscd, amsfonts, mathrsfs, mathtools, tikz}
\usepackage{geometry}
\usepackage{hyperref}
\usepackage{setspace}
\usepackage{microtype}
\usepackage{enumitem}
\usepackage{float}

\newtheorem{theorem}{Theorem}[section]
\newtheorem{lemma}[theorem]{Lemma}

\newtheorem{corollary}[theorem]{Corollary}
\theoremstyle{definition}

\newcommand{\Z}{\mathbb{Z}}
\newcommand{\Q}{\mathbb{Q}}
\newcommand{\R}{\mathbb{R}}
\newcommand{\C}{\mathbb{C}}
\newcommand{\F}{\mathbb{F}}
\newcommand{\G}{\mathbb{G}}
\newcommand{\OT}{\mathcal{O}_T}

\newcommand{\bb}[1]{\mathbb{#1}}
\newcommand{\bd}{\mathbf{d}}
\newcommand{\Tr}{\operatorname{Tr}}
\newcommand{\ord}{\operatorname{ord}}
\newcommand{\lcm}{\operatorname{lcm}}
\newcommand{\supp}{\operatorname{supp}}
\newcommand{\Gal}{\operatorname{Gal}}
\newcommand{\Teich}{\operatorname{Teich}}

\newcommand{\HP}{\operatorname{HP}}
\newcommand{\NP}{\operatorname{NP}}
\newcommand{\ALam}{\OT}

\begin{document}

\title{On partial $T$-adic exponential sums \\ and partial exponential sums with $p$-power conductor}
\author{C. Douglas Haessig}
\date{\today}

\maketitle

\begin{abstract}
Liu and Wan \cite{Liu_2009} introduced $T$-adic exponential sums as a way to interpolate all character sums with character $\psi$ having $p$-power conductor. In this paper, we generalize their $T$-adic theory to partial $T$-adic exponential sums. We prove that the associated $L$-functions are $T$-adic meromorphic, and as a consequence, give a $p$-adic proof of rationality for all partial $L$-functions of characters $\psi$. We also give Newton-over-Hodge estimates.
\end{abstract}

\tableofcontents
\vspace{1em}

\section{Introduction}

Let $p$ be a prime and let $q = p^a$. Denote by $\Q_q$ the unramified extension of $\Q_p$ of degree $a$, and let $\Z_q$ be its ring of integers. Consider a Laurent polynomial
\[
f(x) = \sum_{u \in \supp(f)} c_u x^u \in \Z_q[x_1^{\pm 1}, \dots, x_n^{\pm 1}]
\]
whose coefficients $c_u$ are Teichm\"uller units in $\Z_q$.

We fix $\bd = (d_1, \dots, d_n) \in \Z_{\geq 1}^n$, and set $d := \lcm(d_1, \dots, d_n)$. For any $r \geq 1$, denote by $\mu_{q^r - 1}$ the group of $(q^r - 1)$-th roots of unity.  For each $k \geq 1$, define the \emph{partial $T$-adic exponential sum} by
\[
S_\bd(f, k; T) := \sum (1 + T)^{\Tr_{\Q_{q^{kd}}/\Q_p}(f(x_1, \dots, x_n))} \;\in\; \Z_p[[T]],
\]
where the sum runs over $x_i \in \mu_{q^{kd_i} - 1} \subset \Q_{q^{kd}}$ for $i = 1, \dots, n$. The associated \emph{partial $T$-adic $L$-function} is defined by
\[
L_\bd(f, T; s) := \exp\!\left( \sum_{k=1}^\infty S_\bd(f, k; T) \, \frac{s^k}{k} \right).
\]

The $T$-adic $L$-function interpolates all partial exponential sums of additive characters of arbitrary $p$-power conductor. Let us explain. Let $\psi: \Z_p \to \C_p^\times$ be a continuous, locally constant, additive character of conductor $p^m$. Setting $\pi_\psi := \psi(1) - 1$, then $\psi$ has the form $\psi(z) = (1 + \pi_\psi)^z$ for all $z \in \Z_p$. For $k \geq 1$, define the partial exponential sum
\[
S_{\bd, \psi}(f, k) := \sum \psi \circ \Tr_{\Q_{q^{kd}}/\Q_p}\big(f(x_1, \dots, x_n)\big),
\]
where the sum runs over $x_i \in \mu_{q^{kd_i} - 1} \subset \Q_{q^{kd}}$ for $1 \leq i \leq n$. The associated $L$-function is
\[
L_{\bd, \psi}(f; s) := \exp\!\left( \sum_{k=1}^\infty S_{\bd, \psi}(f, k) \, \frac{s^k}{k} \right).
\]
By construction, specialization of the $T$-adic partial exponential sum at $T = \pi_\psi$ recovers the partial exponential sum:
\[
S_{\bd, \psi}(f, k) = S_\bd(f, k; T)\big|_{T = \pi_\psi} \qquad \text{and} \qquad L_{\bd, \psi}(f; s) = L_\bd(f, T; s)\big|_{T = \pi_\psi}.
\]

Our main results are the following, which is proven in pieces throughout the paper.

\begin{theorem}\label{thm:main_meromorphy}
The partial $T$-adic $L$-function $L_\bd(f, T; s)$ is $T$-adic meromorphic over $\Z_p[[T]]$. Specifically, there exist $T$-adic entire $A(T, s), B(T, s) \in 1 + s \Z_p[[T]][[s]]$ such that $L_\bd(f, T; s) = \frac{A(T, s)}{B(T, s)}$. The $T$-adic Newton polygons of $A(T, s)$ and $B(T, s)$ are both bounded below by $a(p-1) \cdot \HP_\bd(\Delta)^{\oplus 2^{n-1}}$; see Section \ref{sec:noh_L_func} for the definition of the lower bound.
\end{theorem}

\begin{theorem}\label{thm:main_rationality}
Let $\psi$ be a continuous, locally constant, additive character on $\Z_p$ of conductor $p^m$. Then the partial $L$-function is a rational function $L_{\bd, \psi}(f; s) \in \Q(\zeta_{p^m})(s)$. Writing $L_{\bd, \psi}(f; s)^{(-1)^{n-1}} = A(s) / B(s)$ in reduced form for $A, B \in 1+s\Z[\zeta_{p^m}][s]$, then their $q$-adic Newton polygons satisfy
\[
\NP_q(A) \;\ge\; \HP_{\bd, \psi}^+(\Delta) \qquad \text{and} \qquad \NP_q(B) \;\ge\; \HP_{\bd, \psi}^-(\Delta).
\]
See Section \ref{sec:noh_L_func} for the definition of $\HP_{\bd, \psi}^\pm(\Delta)$.
\end{theorem}

Partial $L$-functions and partial zeta functions were introduced by Wan in \cite{MR1803946}, and their rationality proven in \cite{MR2027782}.  Further results were given by Fu and Wan using $\ell$-adic techniques in \cite{MR1960121, MR2030375}. For a $p$-adic study, see \cite{partial_Haessig} and \cite{MR4531526}. The partial zeta function of Fermat hypersurfaces was studied in \cite{MR4318328}. See also \cite{MR2385246} for moment zeta functions of toric Calabi-Yau hypersurfaces.

\(T\)-adic exponential sums were introduced by Liu and Wan in \cite{Liu_2009}. They played a critical role in the study of slope variation in Artin--Schreier--Witt towers; see  \cite{Davis_2015, MR3814335, Ren_2019}. For further results on $T$-adic exponential sums and their multiplicative twists, see \cite{Liu_2016, Liu_2014,Niu_2018,Schmidt_2023}.

\section{Integrality via Euler product}\label{sec:euler}

In this section, we give an Euler product expression for $L_\bd(f, T; s)$. As a consequence, $L_\bd(f, T; s)  \in 1 + s \Z_p[[T]][[s]]$. 

Denote by $\G_m^n/\F_q$ the $n$-fold multiplicative torus over $\F_q$. Let $\bar x = (\bar x_1, \ldots, \bar x_n) \in |\G_m^n / \F_q|$ be a closed point. In the following, since most constructions will only depend on the closed point, we will abuse notation and view $\bar x$ as both a closed point and as an element $\bar x = (\bar x_1, \ldots, \bar x_n)$ in the Galois orbit. Define the degree of $\bar x$ by $\deg(\bar x) := [\F_q(\bar x_1, \ldots, \bar x_n) : \F_q]$. Set $\deg(\bar x_i) := [\F_q(\bar x_i): \F_q]$. Then 
\[
\deg(\bar x) = \lcm( \deg(\bar x_1), \ldots, \deg(\bar x_n) ) .
\]
Define the partial degree of $\bar x$ with respect to $\bd$ as
\[
\deg_\bd(\bar x) := \lcm_{1 \leq i \leq n}\!\left( \frac{\deg(\bar x_i)}{\gcd(\deg(\bar x_i), d_i)} \right).
\]

\begin{lemma}\label{lem: degree stuff}
For every $\bar x \in |\G_m^n / \F_q|$,
\begin{enumerate}[label=\textup{(\roman*)}]
\item $\deg_\bd(\bar x) \mid \deg(\bar x)$ and $\deg(\bar x) \mid ( \deg_\bd(\bar x) \cdot d)$.
\item $\bar x = (\bar x_1, \ldots, \bar x_n) \in \prod_{i=1}^n \F_{q^{k d_i}}^\times$ if and only if $\deg_\bd(\bar x) \mid k$.
\end{enumerate}
\end{lemma}

\begin{proof}
The proof of the first assertion is straightforward. For the second, note that $\bar x \in \prod_{i=1}^n \F_{q^{k d_i}}^\times$ is equivalent to $\deg(\bar x_i) \mid k d_i$ for every $1 \leq i \leq n$, and this is equivalent to $\frac{\deg(\bar x_i)}{\gcd(\deg(\bar x_i), d_i)} \mid k$. This holds for all $i$ if and only if their least common multiple, $\deg_\bd(\bar x)$, divides $k$.
\end{proof}

For $\bar x \in |\G_m^n / \F_q|$, denote its Teichm\"uller lift by $\hat{x} = (\hat{x}_1, \dots, \hat{x}_n) \in \Q_{q^{\deg(\bar x)}}^n$. To ease notation, set
\[
t(\bar x) := \Tr_{\Q_{q^{\deg(\bar x)}}/\Q_p}\!\big( f(\hat{x}) \big) \;\in\; \Z_p \qquad \text{and} \qquad e(\bar x) := \frac{\deg_\bd(\bar x) \cdot d}{\deg(\bar x)} \;\in\; \Z_{\geq 1}.
\]

\begin{lemma}\label{lem:trace_inflation}
Let $\bar x \in |\G_m^n / \F_q|$ and let $k \geq 1$ be an integer such that $\deg_\bd(\bar x) \mid k$. Then $\deg(\bar x) \mid kd$, and
\[
\Tr_{\Q_{q^{kd}}/\Q_p}\!\big( f(\hat{x}) \big) = e(\bar x) \cdot \frac{k}{\deg_\bd(\bar x)} \cdot t(\bar x).
\]
\end{lemma}

\begin{proof}
By Lemma~\ref{lem: degree stuff}(i) and the hypothesis $\deg_\bd(\bar x) \mid k$, we have $\deg(\bar x) \mid \deg_\bd(\bar x) d \mid kd$. Thus, the field extension $\Q_{q^{kd}} / \Q_{q^{\deg(\bar x)}}$ is well-defined. Since $f(\hat{x}) \in \Q_{q^{\deg(\bar x)}}$,
\[
\Tr_{\Q_{q^{kd}}/\Q_p}\!\big(f(\hat{x})\big) = \Tr_{\Q_{q^{\deg(\bar x)}}/\Q_p}\!\left( \Tr_{\Q_{q^{kd}}/\Q_{q^{\deg(\bar x)}}}\!\big(f(\hat{x})\big) \right) = \frac{kd}{\deg(\bar x)} \cdot \Tr_{\Q_{q^{\deg(\bar x)}}/\Q_p}\!\big(f(\hat{x})\big).
\]
The lemma follows.
\end{proof}

\begin{theorem}\label{thm:euler}
We have the Euler product:
\[
L_\bd(f, T; s) = \prod_{\bar x \in |\G_m^n / \F_q|} \left( 1 - (1 + T)^{e(\bar x) t(\bar x)} \, s^{\deg_\bd(\bar x)} \right)^{- \deg(\bar x)/\deg_\bd(\bar x)}.
\]
Hence, $L_\bd(f, T; s) \;\in\; 1 + s\Z_p[[T]][[s]]$.
\end{theorem}

\begin{proof}
From Lemmas \ref{lem: degree stuff} and \ref{lem:trace_inflation}, we may sum over closed points:
\[
S_\bd(f, k; T) = \sum_{\substack{\bar x \in |\G_m^n / \F_q| \\ \deg_\bd(\bar x) \mid k}} \deg(\bar x) \cdot (1 + T)^{e(\bar x) t(\bar x) \cdot \frac{k}{\deg_\bd(\bar x)}}.
\]
Then
\begin{align*}
\log L_\bd(f, T; s) &= \sum_{k=1}^\infty S_\bd(f, k; T) \frac{s^k}{k} \\
&= \sum_{\bar x \in |\G_m^n / \F_q|} \deg(\bar x) \sum_{j=1}^\infty \frac{1}{j \cdot \deg_\bd(\bar x)} \left( (1 + T)^{e(\bar x) t(\bar x)} s^{\deg_\bd(\bar x)} \right)^j,
\end{align*}
where we have reindexed the inner sum by setting $k = j \cdot \deg_\bd(\bar x)$. Thus,
\begin{align*}
\log L_\bd(f, T; s) &= \sum_{\bar x \in |\G_m^n / \F_q|} \frac{\deg(\bar x)}{\deg_\bd(\bar x)} \sum_{j=1}^\infty \frac{1}{j} \left( (1 + T)^{e(\bar x) t(\bar x)} s^{\deg_\bd(\bar x)} \right)^j \\
&= -\sum_{\bar x \in |\G_m^n / \F_q|} \frac{\deg(\bar x)}{\deg_\bd(\bar x)} \log\!\left( 1 - (1 + T)^{e(\bar x) t(\bar x)} s^{\deg_\bd(\bar x)} \right).
\end{align*}
The result follows.
\end{proof}

\section{Partial $T$-adic $L$-functions and $C$-functions}\label{sec:Tadic}

This section closely follows the method developed in \cite{partial_Haessig}; we refer to that paper for motivation for the following construction. Set $N := \sum_{i=1}^n d_i$. We introduce $N$ variables $y_{i, j}$ where $1 \leq i \leq n$ and $j \in \Z/d_i\Z$. For $v \in \Z^N$, we use the multi-index notation $y^v := \prod_{i, j} y_{i, j}^{v_{i, j}}$. Define the shift operator $\sigma$ on the variables $y_{i, j}$ by
\[
\sigma(y_{i, j \bmod d_i}) := y_{i, (j+1) \bmod d_i}.
\]
On functions $\xi(y)$, define the action of $\sigma$ by $(\sigma \xi)(y) := \xi(\sigma y)$. Denote by $P$ the $N \times N$ permutation matrix acting on $\Z^N$ by $(Pv)_{i, j} := v_{i, (j-1) \bmod d_i}$. Then $\sigma(y^v) = y^{Pv}$ and $\sigma$ has order $d$.

Define the \emph{unfolded} Laurent polynomial
\[
G(y) := \sum_{j = 0}^{d-1} f\!\left( y_{1, j \bmod d_1}, \dots, y_{n, j \bmod d_n} \right).
\]
Let $\Delta = \Delta(G) \subset \R^N$ denote the Newton polyhedron of $G$ at infinity, and set
\[
C(\Delta) := \R_{\geq 0} \Delta, \qquad M := C(\Delta) \cap \Z^N.
\]
Let $w: M \to \Q_{\geq 0}$ be the associated polyhedral weight function, and let $D$ be the smallest positive integer with $w(M) \subseteq \tfrac{1}{D} \Z$. By \cite[Lemma 2.2]{partial_Haessig}, $G$ is $\sigma$-invariant; equivalently, $\supp(G) \subset \Z^N$ is stable under $P$. Thus $P\Delta = \Delta$, and the weight function satisfies $w(Pu) = w(u)$ for all $u \in M$ (see \cite[Lemma 2.3]{partial_Haessig}).

Denote by $E(t) := \exp\!\big( \sum_{r \geq 0} t^{p^r}/p^r \big) \in 1 + t\Z_p[[t]]$ the Artin--Hasse exponential. Define $\pi$ by the relation $E(\pi) = 1 + T$,
and note that $\pi = T + \cdots  \in T \Z_p[[T]]$; thus $\ord_T \pi = 1$. Set $\OT := \Z_q\big[[\pi^{1/D}]\big]$. Observe that $\Z_p[[T]] \subset \OT$ since $T = E(\pi) - 1 \in \pi \Z_p[[\pi]]$.

Define the $T$-adic Banach algebra $B$ attached to $\Delta$ with orthonormal basis $\{\pi^{w(u)} y^u\}_{u \in M}$:
\[
B := \left\{ \sum_{u \in M} A_u \, \pi^{w(u)} y^u \;\middle|\; A_u \in \OT, \;\; \ord_T(A_u) \to \infty \text{ as } w(u) \to \infty \right\}.
\]
Define 
\[
F(y) := \prod_{j=0}^{d-1} \prod_{u \in \supp(f)} E(\pi c_u \prod_{i=1}^n y_{i, j \bmod d_i}^{u_i}).
\]
Choose $\tau \in \Gal(\Q_q/\Q_p)$ such that on Teichm\"uller units, $\tau(a) = a^p$. Set
\[
F_a(y) := F(y) F^{\tau}(y^p) \cdots F^{\tau^{a-1}}(y^{p^{a-1}}).
\]
Define $\psi_p$ by its action on monomials: $\psi_p(y^u) = y^{u/p}$ if every coordinate of $u$ is divisible by $p$, and $0$ otherwise. Set $\psi_q := \psi_p^a$. Define the Frobenius operator
\[
\alpha_a := \psi_q \circ F_a(y).
\]

\begin{lemma}\label{lem:weight_estimate}
Write $F(y) = \sum_{u \in M} B_u y^u$. For all $u \in M$,
\[
\ord_T(B_u) \geq w(u).
\]
\end{lemma}

\begin{proof}
The usual argument shows the coefficients of $E(\pi t) = \sum_{i=0}^\infty \lambda_i t^i$ satisfy $\ord_T(\lambda_i) \geq i$. Write $F(y)$ in the form $F(y) = \prod_{k=1}^m E(\pi a_k y^{v^{(k)}})$ where $a_k$ is a Teichm\"uller unit, and $v^{(k)} \in \supp(G)$. Note, $w(v^{(k)}) \leq 1$.

Writing $F(y) = \sum_{u \in M} B_u y^u$, then the coefficient $B_u$ is a sum of terms of the form $\lambda_{i_1} \cdots \lambda_{i_m} a_1^{i_1} \cdots a_m^{i_m}$ where $i_1 {v}^{(1)} + \dots + i_m {v}^{(m)} = u$. Now,
\[
\ord_T(\lambda_{i_1} \cdots \lambda_{i_m}) \geq \sum_{k=1}^m i_k \;\geq\; \sum_{k=1}^m i_k w(v^{(k)}) \;\geq\; w\!\left( \sum_{k=1}^m i_k v^{(k)} \right) = w(u).
\]
Consequently, $\ord_T(B_u) \geq w(u)$ as desired. 
\end{proof}

\begin{theorem}\label{thm:complete_continuity}
The Frobenius operator $\alpha_a$ is a well-defined, completely continuous endomorphism of $B$.
\end{theorem}

\begin{proof}
Observe that $\alpha_a = \alpha^a$, where $\alpha := \tau^{-1} \circ \psi_p \circ F(y)$. Thus, it suffices to prove that $\alpha$ is completely continuous on $B$. Now, $\alpha$ acts on the orthonormal basis $\{ \pi^{w(u)} y^u \}_{u \in M}$ of $B$ by: 
\begin{align*}
\alpha\!\left(\pi^{w(u)} y^u\right) &= \tau^{-1} \circ \psi_p \!\left( \sum_{v \in M} B_v \pi^{w(u)} y^{u+v} \right) \\
&= \sum_{r \in M} \tau^{-1}(B_{pr-u}) \pi^{w(u)} y^r \\
&= \sum_{r \in M} A_{u,r} \pi^{w(r)} y^r,
\end{align*}
where $A_{u,r} := \tau^{-1}(B_{pr-u}) \pi^{w(u) - w(r)}$. By Lemma~\ref{lem:weight_estimate},
\begin{align*}
\ord_T(A_{u,r}) &= \ord_T(B_{pr-u}) + w(u) - w(r) \\
&\geq w(pr-u) + w(u) - w(r)\\
&\geq w(pr) - w(u) + w(u) - w(r) \\
&= (p-1)w(r).
\end{align*}
The result follows.
\end{proof}

\begin{lemma}\label{lem:commutation}
As operators on $B$, $\sigma \circ \alpha_a = \alpha_a \circ \sigma$.
\end{lemma}

\begin{proof}
Since $G$ is $\sigma$-invariant, $F$ and $F_a$ are $\sigma$-invariant. The result follows.
\end{proof}

\begin{lemma}\label{lem:det_identity}
Let $P$ be the permutation matrix associated with the shift operator $\sigma$. For any integer $k \geq 1$, we have
\[
\det(q^k I - P) = \prod_{i=1}^n (q^{k d_i} - 1).
\]
\end{lemma}

\begin{proof}
This is \cite[Lemma 4.1]{partial_Haessig}.
\end{proof}

\begin{lemma}\label{lem:t_adic_splitting}
For any Teichm\"uller unit $z \in \Z_q$,
\[
\prod_{j=0}^{a-1} E(\pi z^{p^j}) = (1 + T)^{\Tr_{\Q_q/\Q_p}(z)}.
\]
\end{lemma}

\begin{proof}
Taking the log of the left-hand side:
$$ \log \prod_{j=0}^{a-1} E(\pi z^{p^j}) = \sum_{j=0}^{a-1} \sum_{i=0}^\infty \frac{\pi^{p^i} z^{p^{i+j}}}{p^i} = \sum_{i=0}^\infty \frac{\pi^{p^i}}{p^i} \left( \sum_{j=0}^{a-1} z^{p^{i+j}} \right). $$
Since $z$ is a Teichmüller unit in $\Z_q$, it satisfies $z^{p^a} = z$. Thus, for any fixed $i \geq 0$, as $j$ varies from $0$ to $a-1$, the exponent $i+j \pmod a$ is a permutation of the integers $0$ to $a-1$. Hence, $\sum_{j=0}^{a-1} z^{p^{i+j}} = \Tr_{\Q_q/\Q_p}(z)$. Thus,
$$ \Tr_{\Q_q/\Q_p}(z) \sum_{i=0}^\infty \frac{\pi^{p^i}}{p^i} = \Tr_{\Q_q/\Q_p}(z) \log E(\pi) = \Tr_{\Q_q/\Q_p}(z) \log(1+T), $$
as desired.
\end{proof}

Define
\[
W_k := \big\{ y \in (\overline{\bb F}_p^\times)^N \;\big|\; \sigma^{-1}(y^{q^k}) = y \big\}.
\]
Set $F_a^{(k)}(y) := \prod_{j=0}^{k-1} F_a(y^{q^j})$.

\begin{lemma}
We have
\[
S_\bd(f, k; T) = \sum_{\substack{\bar y \in W_k \\ \hat y = \Teich(\bar y)}} F_a^{(k)}(\hat y).
\]
\end{lemma}

\begin{proof}
The proof is identical to \cite[Theorem 4.2]{partial_Haessig}.
\end{proof}

\begin{theorem}[$T$-adic Dwork Trace Formula]\label{thm:trace}
For every integer $k \geq 1$,
\[
S_\bd(f, k; T) = \det(q^k I - P) \cdot \Tr\!\left( \sigma \circ \alpha_a^k \,\big|\, B \right) = \left( \prod_{i=1}^n (q^{k d_i} - 1) \right) \cdot \Tr\!\left( \sigma \circ \alpha_a^k \,\big|\, B \right).
\]
\end{theorem}

\begin{proof}
The proof is identical to \cite[Theorem 4.4]{partial_Haessig}. 
\end{proof}

In the $T$-adic theory of Liu and Wan \cite{Liu_2009}, great use is made of an auxiliary $C$-function. Motivated by this, we define the partial $T$-adic $C$-function:
\begin{align*}
C_\bd(f, T; s) &:= \exp\!\left( -\sum_{k=1}^\infty \frac{S_\bd(f, k; T)}{\prod_{i=1}^n (q^{k d_i} - 1)} \, \frac{s^k}{k} \right) \\
&=  \exp\!\left( -\sum_{k = 1}^\infty \Tr\!\left( \sigma \circ \alpha_{a}^k \,\big|\, B \right) \frac{s^k}{k} \right) \qquad 
\text{(Theorem~\ref{thm:trace}).}
\end{align*}
Define the \emph{$\sigma$-twisted Fredholm determinant}
\[
{\det}_\sigma\!\left( I - s\alpha_a \,\big|\, B \right) := \exp\!\left( -\sum_{k=1}^\infty \Tr\!\left( \sigma \circ \alpha_a^k \,\big|\, B \right) \frac{s^k}{k} \right).
\]
With this notation,
\[
C_\bd(f, T; s) = {\det}_\sigma\!\left( I - s\alpha_a \,\big|\, B \right).
\]

We warn the reader that, even though $\alpha_a$ is completely continuous, a $\sigma$-twisted Fredholm determinant need not be $T$-adic entire in general; cf.~\cite[Section 5]{partial_Haessig}. However, in our current situation, it will be $T$-adic meromorphic, as shown in Theorem \ref{thm:Cd_meromorphic} below. Before moving to meromorphy let us give the relation with the $L$-function. For $r \geq 1$, define the operator $\delta_r$ by $F(s)^{\delta_r} := F(s)/F(q^r s)$. For $\bd = (d_1, \dots, d_n)$ define $\delta_\bd := \delta_{d_1} \circ \cdots \circ \delta_{d_n}$.

\begin{theorem}\label{thm:L_C_relation}
\[
L_\bd(f, T; s)^{(-1)^{n-1}} = C_\bd(f, T; s)^{\delta_\bd}.
\]
\end{theorem}

\begin{proof}
The proof is identical to that of \cite[Lemma 6.1]{partial_Haessig}.
\end{proof}

\begin{corollary}\label{cor:Cd_integral}
\[
C_\bd(f, T; s) \;\in\; 1 + s \Z_p[[T]][[s]].
\]
\end{corollary}

\begin{proof}
Observe that the inverse $\delta_r^{-1}$ of $\delta_r$ is given by $F(s)^{\delta_r^{-1}} = \prod_{j \geq 0} F(q^{jr} s)$. Thus $\delta_\bd^{-1} = \delta_{d_n}^{-1} \circ \cdots \circ \delta_{d_1}^{-1}$ gives
\begin{align}\label{equ: C and L inverse relation}
C_\bd(f, T; s)^{(-1)^{n-1}} &= L_\bd(f, T; s)^{\delta_\bd^{-1}} \notag \\
&= \prod_{(j_1, \dots, j_n) \in \Z_{\geq 0}^n} L_\bd\!\left( f, T; q^{j_1 d_1 + \cdots + j_n d_n} s \right).
\end{align}
By Theorem~\ref{thm:euler}, every factor lies in $1 + s\Z_p[[T]][[s]]$, proving the result.
\end{proof}

Denote by $K := \operatorname{Frac}(\OT)$ the fraction field of $\OT$, and let $\Omega$ be the completion of an algebraic closure of
$K$. By Theorem~\ref{thm:complete_continuity} the Frobenius $\alpha_a$ is completely continuous on
$B$. Thus, every generalized eigenspace
\[
V_\lambda := \bigcup_{j \ge 1} \ker\big((\alpha_a - \lambda)^j \,\big|\, B \widehat{\otimes}_{\mathcal{O}_T} \Omega\big)
\qquad (\lambda \neq 0)
\]
is finite-dimensional, and for every real $c$ only finitely many eigenvalues satisfy
$\ord_T \lambda \le c$; in particular $\ord_T \lambda \to \infty$. By Lemma~\ref{lem:commutation}
the shift operator $\sigma$ commutes with $\alpha_a$, and so each $V_\lambda$ is $\sigma$-stable. Set
$m_\lambda(\sigma) := \Tr(\sigma \mid V_\lambda)$. Since $\sigma^d = 1$, we have $m_\lambda(\sigma) \in \Z[\zeta_d]$. By the same argument as in \cite[Theorem~7.1]{partial_Haessig}, we have $m_\lambda(\sigma) \in \Z$. 

\begin{lemma}\label{lem: C product over eigenvalues}
\[
C_\bd(f, T; s) = \prod_{\lambda \neq 0} (1 - \lambda s)^{m_\lambda(\sigma)} \in 1 + s\,\Omega[[s]].
\]
\end{lemma}

\begin{proof}
For $\lambda \neq 0$, write $\alpha_a|_{V_\lambda} = \lambda I + N_\lambda$ with
$N_\lambda$ nilpotent. Since $\alpha_a$ commutes with $\sigma$, so does $N_\lambda$. Hence, $\sigma N_\lambda^{\,j}$ is nilpotent for $j \ge 1$. Thus, 
$\Tr(\sigma \circ \alpha_a^k \mid V_\lambda) = \lambda^k\, m_\lambda(\sigma)$, and so for all $k \geq 1$:
\[
\Tr\big(\sigma \circ \alpha_a^k \,\big|\, B\big) = \sum_{\lambda \neq 0} m_\lambda(\sigma)\, \lambda^k.
\]
Substituting this into the definition of the $\sigma$-twisted Fredholm determinant gives the result.
\end{proof}

\begin{theorem}\label{thm:Cd_meromorphic}
The partial $T$-adic $C$-function $C_\bd(f, T; s)$ is $T$-adic meromorphic over $\Z_p[[T]]$. Specifically,
there exist $T$-adic entire $A(T, s),\, B(T, s) \in 1 + s\,\Z_p[[T]][[s]]$ such that
\[
C_\bd(f, T; s) = \frac{A(T, s)}{B(T, s)}.
\]
\end{theorem}

\begin{proof}
Since each $m_\lambda(\sigma) \in \Z$, define
\[
A(T, s) := \prod_{m_\lambda(\sigma) > 0} (1 - \lambda s)^{m_\lambda(\sigma)} \qquad \text{and} \qquad
\qquad
B(T, s) := \prod_{m_\lambda(\sigma) < 0} (1 - \lambda s)^{-m_\lambda(\sigma)}.
\]
Thus, $C_\bd = A/B$ with $A, B \in 1 + s\,\Omega[[s]]$. Note that $A$ and $B$ are coprime and $T$-adic entire since $\ord_T \lambda \to \infty$.

We are left with showing $A$ and $B$ lie in $1 + s\,\Z_p[[T]][[s]]$. Set $K_0 := \operatorname{Frac}(\Z_p[[T]])$. Since the coefficients of $C_\bd$ lie in $\Z_p[[T]]$, any automorphism $g \in \operatorname{Aut}(\Omega / K_0)$ acts trivially on $C_{\mathbf{d}}$. Thus, $C_{\mathbf{d}} = g(A)/g(B)$. Since \(g\) preserves the \(T\)-adic valuation, we see that $g(A)$ and $g(B)$ are $T$-adic entire, coprime, and give a reduced fraction of $C_\bd$. Hence, $g(A) = A$ and $g(B) = B$. Consequently, the coefficients of $A$ and $B$ lie in $K_0$. 

Next, since the Fredholm determinant of $\alpha_a$ on $B$ lies in $1 + s \mathcal{O}_T[[s]]$, each $\lambda$ is integral over $\mathcal{O}_T$. Hence, the coefficients of $A$ and $B$ are integral over $\mathcal{O}_T$ and lie in $K_0$. Since $\mathcal{O}_T$ is integral over $\Z_p[[T]]$, and since \(\Z_p[[T]]\) is integrally closed in its fraction field \(K_0\), we have $A, B \in 1+s\,\Z_p[[T]][[s]]$.
\end{proof}

\begin{corollary}\label{cor: L is T-adic mero}
The partial $T$-adic $L$-function $L_\bd(f, T; s)$ is $T$-adic meromorphic over $\Z_p[[T]]$. Specifically,
there exist $T$-adic entire $A(T, s),\, B(T, s) \in 1 + s\,\Z_p[[T]][[s]]$ such that
\[
L_\bd(f, T; s) = \frac{A(T, s)}{B(T, s)}.
\]
\end{corollary}

\begin{proof}
This follows from Theorems \ref{thm:Cd_meromorphic} and \ref{thm:L_C_relation}.
\end{proof}

\section{Rationality of specialization}\label{sec:rationality}

We now move to partial toric exponential sums of conductor a $p$-power. Let $\psi: \Z_p \to \C_p^\times$ be a continuous, locally constant, additive character of conductor $p^m$. Setting $\pi_\psi := \psi(1) - 1$, then $\psi(z) = (1 + \pi_\psi)^z$ for all $z \in \Z_p$. The partial exponential sums are defined by: for each $k \geq 1$,
\[
S_{\bd, \psi}(f, k) := \sum \psi \circ \Tr_{\Q_{q^{kd}}/\Q_p}\big(f(x_1, \dots, x_n)\big),
\]
where the sum runs over $x_i \in \mu_{q^{kd_i} - 1}$ for $1 \leq i \leq n$. The associated partial $L$-function is
\[
L_{\bd, \psi}(f; s) := \exp\!\left( \sum_{k=1}^\infty S_{\bd, \psi}(f, k) \, \frac{s^k}{k} \right).
\]

We define a specialization map $\mathrm{sp}_\psi$ as follows.

\begin{lemma}\label{lem:specialization}
There exists $\gamma_\psi \in \Z_p[\zeta_{p^m}]$ satisfying $E(\gamma_\psi) = 1 + \pi_\psi$, with $\ord_{\pi_\psi}(\gamma_\psi)  = 1$. The map $\pi \mapsto \gamma_\psi$ extends uniquely to a continuous ring homomorphism
\[
\mathrm{sp}_\psi: \ALam = \Z_q\big[[\pi^{1/D}]\big] \;\longrightarrow\; \Z_q[\zeta_{p^m}, \gamma_\psi^{1/D}],
\]
and this map sends $T = E(\pi) - 1 \mapsto \pi_\psi$. Consequently, the specialization $T \mapsto \pi_\psi$ is a well-defined continuous homomorphism from $\Z_p[[T]]$ to $\Z_q[\zeta_{p^m}, \gamma_\psi^{1/D}]$, and for every $h \in \Z_p[[T]]$,
\[
\ord_{\pi_\psi}\!\big( \mathrm{sp}_\psi(h) \big) \;\geq\; \ord_T(h).
\]
\end{lemma}

\begin{proof}
Since the Artin--Hasse exponential $E(t) \in 1 + t\Z_p[[t]]$ satisfies $E(t) - 1 = t + O(t^2)$, there exists a unique formal inverse function $H(T) \in T\Z_p[[T]]$ such that $E(H(T)) = 1 + T$. Since $\ord_p \pi_\psi > 0$, the series $H(T)$ converges at $T = \pi_\psi$. Set $\gamma_\psi := H(\pi_\psi) \in \Z_p[\zeta_{p^m}]$. By construction, $E(\gamma_\psi) = 1 + \pi_\psi$. Furthermore, $\gamma_\psi \equiv \pi_\psi \pmod{\pi_\psi^2}$, and so $\ord_{\pi_\psi}(\gamma_\psi) = 1$.

For the inequality, let $h(T) = \sum_{j \geq 0} h_j T^j \in \Z_p[[T]]$ and set $r := \ord_T(h)$. This means $h(T) = T^r u(T)$ for some $u \in \Z_p[[T]]$ with $T \nmid u(T)$. Specializing, we have $\mathrm{sp}_\psi(h) = \pi_\psi^r \cdot u(\pi_\psi)$. Since $u(\pi_\psi)$ converges to an element in $\Z_p[\zeta_{p^m}]$, $\ord_{\pi_\psi} u(\pi_\psi) \ge 0$. Thus $\ord_{\pi_\psi} \mathrm{sp}_\psi(h) \;\geq\; r = \ord_T(h)$.
\end{proof}

We extend the action of the specialization map to series $H \in 1 + s\Z_p[[T]][[s]]$ by specializing the coefficients. Thus, $H_\psi(s) := H(s)|_{T = \pi_\psi} \in 1 + s\Z_p[\zeta_{p^m}][[s]]$. In particular,
\[
C_{\bd, \psi}(f; s) = C_\bd(f, T; s)\big|_{T = \pi_\psi}, \qquad L_{\bd, \psi}(f; s) = L_\bd(f, T; s)\big|_{T = \pi_\psi}.
\]
We may now prove rationality. 

\begin{theorem}\label{thm: rationality}
$L_{\bd, \psi}(f; s) \in \Q(\zeta_{p^m})(s)$.
\end{theorem}

\begin{proof}
We first prove $p$-adic meromorphy via specialization. By Corollary \ref{cor: L is T-adic mero}, we may write $L_\bd(f, T; s) = A(T, s) / B(T, s)$, where $A(T, s) = \sum_{j \geq 0} a_j(T) s^j$ and $B(T, s)$ are $T$-adic entire series with coefficients in $\Z_p[[T]]$. By Lemma~\ref{lem:specialization}, for every $j \geq 0$:
\[
\ord_{\pi_\psi}\!\big( \mathrm{sp}_\psi(a_j) \big) \;\geq\; \ord_T(a_j)
\]
Consequently, $A(\pi_\psi, s)$ is $\pi_\psi$-adic entire in $s$. Similar for $B(\pi_\psi, s)$. Hence,
\[
L_{\bd, \psi}(f; s) = L_\bd(f, T; s)\big|_{T = \pi_\psi} = \frac{A(\pi_\psi, s)}{B(\pi_\psi, s)}
\]
is $p$-adic meromorphic.

We now have that $L_{\bd, \psi}(f; s)$ is a power series with coefficients in a finite extension field of $\Q$, it is $p$-adic meromorphic, and it is complex analytic in a nonzero disk around the origin since 
\[
|S_{\bd,\psi}(f,k)|_{\C} \leq \prod_i (q^{kd_i}-1).
\]
By Dwork's rationality criterion \cite{Dwork-rationalityofzeta-1960}, these three properties imply $L_{\bd, \psi}(f; s) \in \Q(\zeta_{p^m})(s)$.
\end{proof}

\section{Structure of the $C$-function}\label{sec:meromorphy}

In this section we refine Lemma \ref{lem: C product over eigenvalues}, which wrote the $C$-function as a product over eigenvalues. This refinement will be helpful in Section~\ref{sec:noh} when we study Newton-over-Hodge estimates. In the following, $\varphi$ is Euler's totient and $\mu$ the M\"obius function. 

We begin with some technical lemmas. For any integer $r$, define the set 
\[
W_{k,r} := \{ \bar y \in (\overline{\F}_q^\times)^N \mid \sigma^{-r}(\bar y^{q^k}) = \bar y \}.
\] 

\begin{lemma}\label{lem: G well-defined}
For every $\bar y \in W_{k,r}$, its Teichm\"uller lift $\hat y$ satisfies $G(\hat y) \in \Q_{q^k}$.
\end{lemma}

\begin{proof}
Let $\bar y \in W_{k,r}$. By definition, $\bar y^{q^k} = \sigma^r \bar y$. Let $\hat y$ be the Teichm\"uller lift of $\bar y$, then $\hat y$ also satisfies $\hat y^{q^k} = \sigma^r \hat y$. 

Let $\tau \in \text{Gal}(\mathbb{Q}_p^{\text{unr}} / \mathbb{Q}_p)$ be the Frobenius automorphism. Then $\tau^{ak}$ acts on Teichmüller units by the $q^k$-power map. Since the coefficients of $G$ lie in $\mathbb{Z}_q$, they are fixed by $\tau^a$. Thus 
\[
\tau^{ak} G(\hat{y}) = G(\hat{y}^{q^k}) = G(\sigma^r \hat{y}) = G(\hat y),
\]
where the last equality follows from $G$ being $\sigma$-invariant. This proves the lemma.
\end{proof}

We may now define the partial exponential sum evaluated over $W_{k,r}$, which is well-defined by Lemma \ref{lem: G well-defined}:
\[
S_\bd^{(r)}(f, k; T) := \sum_{\substack{\bar y \in W_{k,r} \\ \hat y = \operatorname{Teich}(\bar y)}} (1 + T)^{\Tr_{\Q_{q^k}/\Q_p} (G(\hat y))}.
\]

\begin{lemma}\label{lem: Sbd coprime}
For every integer $b$ coprime to $d$, we have $S_\bd^{(br)}(f, k; T) = S_\bd^{(r)}(f, k; T)$. 
\end{lemma}

\begin{proof}
Since $b$ is coprime to $d = \lcm(d_1, \dots, d_n)$, it is coprime to every $d_i$. Let $b^{-1}$ denote the inverse of $b$ modulo $d$. Note that the coordinate mapping $y \mapsto z$ defined by
\[
z_{i, j} := y_{i, j b^{-1} \bmod d_i}
\]
is a bijection from $W_{k,r}$ to $W_{k,br}$. Next, we evaluate $G$ on the coordinate change: 
\[
G(z) = \sum_{j=0}^{d-1} f\!\left(z_{1, j \bmod d_1}, \dots, z_{n, j \bmod d_n}\right) = \sum_{j=0}^{d-1} f\!\left(y_{1, j b^{-1} \bmod d_1}, \dots, y_{n, j b^{-1} \bmod d_n}\right) = G(y).
\]
The result follows. 
\end{proof}

\begin{lemma}\label{lem: P^r det}
For every $b$ coprime to $d$, we have $\det(q^k I - P^{br}) = \det(q^k I - P^r)$.
\end{lemma}

\begin{proof}
The permutation matrix $P$ acts on the variables $y_{i,j}$ independently for each $i \in \{1, \dots, n\}$. Thus, $P$ decomposes into a block-diagonal matrix, $P = \mathrm{diag}(P_1, \dots, P_n)$, where each $P_i$ is a $d_i \times d_i$ permutation matrix defined by the shift $u_{i, j \bmod d_i} \mapsto u_{i, j-1 \bmod d_i}$, and so $P_i$ has order $d_i$.

For a fixed block $i$, $P_i^r$ decomposes into $\gcd(r, d_i)$ disjoint cycles, each having length $d_i / \gcd(r, d_i)$. By hypothesis, the integer $b$ is coprime to $d$, and thus coprime to $d_i$. Consequently, $\gcd(br, d_i) = \gcd(r, d_i)$, and so $P_i^{br}$ has the same cycle type as $P_i^r$. This means they have identical characteristic polynomials $\det(t I_{d_i} - P_i^{br}) = \det(t I_{d_i} - P_i^r)$. Hence,
\[
\det(tI - P^r) = \prod_{i=1}^n \det(tI_{d_i} - P_i^r) = \prod_{i=1}^n \det(tI_{d_i} - P_i^{rb}) = \det(tI - P^{rb}).
\]
Evaluating at $t = q^k$ gives the result. 
\end{proof}

\begin{lemma}\label{lem:twist_independence}
For every $j \in \Z$, $b \in (\Z/d\Z)^\times$, and $k \ge 1$, we have
\[ 
\Tr\!\big( \sigma^{jb} \circ \alpha_a^k \mid B \big) = \Tr\!\big( \sigma^{j} \circ \alpha_a^k \mid B \big). 
\]
\end{lemma}

\begin{proof}
Generalizing the $T$-adic Dwork trace formula (Theorem~\ref{thm:trace}) using $\sigma^r$ instead of $\sigma$ gives:
\[
S_\bd^{(r)}(f, k; T) = \det\!\big(q^k I - P^r\big)\, \Tr\!\big(\sigma^r \circ \alpha_a^k \mid B\big).
\]
The result follows from Lemmas \ref{lem: Sbd coprime} and \ref{lem: P^r det}.
\end{proof}

We now recall the setup of Lemma \ref{lem: C product over eigenvalues}. Denote by $K := \operatorname{Frac}(\OT)$ the fraction field of $\OT$, and let $\Omega$ be the completion of an algebraic closure of $K$. Define the generalized eigenspace $V_\lambda$ of $\alpha_a$ on $B \widehat{\otimes}_{\mathcal{O}_T} \Omega$:
\[
V_\lambda := \bigcup_{j \ge 1} \ker\big((\alpha_a - \lambda)^j \,\big|\, B \widehat{\otimes}_{\mathcal{O}_T} \Omega\big).
\]
By Lemma~\ref{lem:commutation} the shift $\sigma$ commutes with $\alpha_a$, so each $V_\lambda$ is $\sigma$-stable. Also, $m_\lambda(\sigma) := \Tr(\sigma \mid V_\lambda) \in \Z$. The operator $\sigma$ satisfies $\sigma^d = 1$, and so we may write its eigenspace decomposition
\[
V_\lambda = \bigoplus_{\zeta \in \mu_d} V_\lambda^{(\zeta)}, \qquad d_{\lambda, \zeta} := \dim_\Omega V_\lambda^{(\zeta)}.
\]
Then $\Tr(\sigma^j \mid V_\lambda) = \sum_{\zeta \in \mu_d} d_{\lambda, \zeta}\, \zeta^j$ for all $j \in \Z$.

\begin{lemma}\label{lem:isotypic_balance}
For every nonzero eigenvalue $\lambda$ of $\alpha_a$, the multiplicity $d_{\lambda,\zeta}$ depends only on the order of the root of unity $\zeta$. That is, if $b \in (\Z/d\Z)^\times$ then
$$d_{\lambda,\zeta^b} = d_{\lambda,\zeta}$$
for every $d$-th root of unity $\zeta$.
\end{lemma}

\begin{proof}
By character orthogonality,
\[
d_{\lambda,\zeta} = \frac{1}{d} \sum_{j=0}^{d-1} \Tr(\sigma^j \mid V_\lambda) \zeta^{-j}.
\]
Fix $b \in (\Z/d\Z)^\times$. Reindexing the following sum via $r \equiv bj \pmod d$ gives
\[
d_{\lambda,\zeta^b} = \frac{1}{d} \sum_{j=0}^{d-1} \Tr(\sigma^j \mid V_\lambda) \zeta^{-bj} = \frac{1}{d} \sum_{r=0}^{d-1} \Tr(\sigma^{b^{-1}r} \mid V_\lambda) \zeta^{-r}.
\]
The lemma follows once we show $\Tr(\sigma^{b^{-1}r} \mid V_\lambda) = \Tr(\sigma^r \mid V_\lambda)$ for all $r \in \Z$ and $\lambda \neq 0$.

By Lemma~\ref{lem:twist_independence}, $\Tr(\sigma^{b^{-1}r} \circ \alpha_a^k \mid B) = \Tr(\sigma^r \circ \alpha_a^k \mid B)$ for all $k \ge 1$. Thus, $\Tr(\sigma^{b^{-1}r} \circ \alpha_a^k \mid B_\Omega) = \Tr(\sigma^r \circ \alpha_a^k \mid B_\Omega)$, which equals
\[
\sum_{\lambda \neq 0} \Tr(\sigma^{b^{-1}r} \mid V_\lambda) \lambda^k = \sum_{\lambda \neq 0} \Tr(\sigma^r \mid V_\lambda) \lambda^k.
\]
Set $e_\lambda := \Tr(\sigma^{b^{-1}r} \mid V_\lambda) - \Tr(\sigma^r \mid V_\lambda)$. Then for every $k \geq 1$, we have $\sum_{\lambda \neq 0} e_\lambda \lambda^k = 0$. We will show $e_\lambda = 0$ for every $\lambda$.

Define the generating function using the variable $s$:
\[
\sum_{k=1}^\infty \left( \sum_{\lambda \neq 0} e_\lambda \lambda^k \right) s^{k-1} = 0.
\]
Since $\ord_T \lambda \to \infty$, we may switch the order of summation to get:
\[
\sum_{\lambda \neq 0} \frac{e_\lambda \lambda}{1 - \lambda s} = 0.
\]
Now, fixing a nonzero eigenvalue $\lambda_0$, we rewrite this as:
\[
e_{\lambda_0} \lambda_0 + (1 - \lambda_0 s) \sum_{\lambda \neq 0, \lambda_0} \frac{e_\lambda \lambda}{1 - \lambda s} = 0.
\]
The left-hand side is $T$-adic meromorphic. Evaluating at $s = \lambda_0^{-1}$ gives $e_{\lambda_0} \lambda_0 = 0$, and so $e_{\lambda_0} = 0$. 
\end{proof}

By Lemma~\ref{lem:isotypic_balance} the following are well-defined: for each $m \mid d$ set
\[
c_{m, \lambda} := d_{\lambda, \zeta} \quad (\text{for any } \zeta \text{ of order } m), \qquad \text{and} \qquad E_m(T, s) := \prod_{\lambda \neq 0} (1 - \lambda s)^{c_{m, \lambda}}.
\]
Observe that each $E_m$ is $T$-adic entire.

\begin{theorem}\label{thm:structural_factorization}
Let $\mathcal S_d$ be the set of squarefree divisors of $d$. Then
\[
C_\bd(f, T; s) = \prod_{m \mid d} E_m(T, s)^{\mu(m)}
= \prod_{m \in \mathcal S_d} E_m(T, s)^{\mu(m)}.
\]
\end{theorem}

\begin{proof}
Grouping $\zeta$ by order and using the identity $\sum_{\zeta \text{ of order } m} \zeta =
\mu(m)$,
\[
m_\lambda(\sigma) = \Tr(\sigma \mid V_\lambda)
= \sum_{\zeta \in \mu_d} d_{\lambda, \zeta}\, \zeta
= \sum_{m \mid d} c_{m, \lambda} \!\!\sum_{\zeta \text{ of order } m}\!\! \zeta
= \sum_{m \mid d} \mu(m)\, c_{m, \lambda}.
\]
Then, by Lemma \ref{lem: C product over eigenvalues}, we have
\[
C_\bd(f, T; s) = \prod_{\lambda \neq 0} (1 - \lambda s)^{m_\lambda(\sigma)}
= \prod_{m \mid d}\Big( \prod_{\lambda \neq 0}(1 - \lambda s)^{c_{m, \lambda}} \Big)^{\mu(m)}
= \prod_{m \mid d} E_m(T, s)^{\mu(m)}.
\]
\end{proof}

\section{Polygon arithmetic}\label{sec: polygons}

In the study of $p$-adic and $T$-adic Newton polygons, we frequently stretch and combine polygons. To make these manipulations precise and easy to read, we establish a formal arithmetic of polygons. 

A \emph{polygon} $\mathcal{P}$ is a multiset of segments $S(\mathcal{P}) = \{ (\mu_i, h_i) \}_{i \in I}$ where $\mu_i \in \R$ is the slope and $h_i > 0$ is the horizontal length (or multiplicity) of the segment. To draw the polygon, one orders the segments such that the slopes are non-decreasing ($\mu_1 \le \mu_2 \le \dots$), then concatenate them end-to-end starting from the origin, each having horizontal length $h_i$ and slope $\mu_i$.

Let $\mathcal{P}$ be a polygon, and let $c \in \R$ be a real scalar. We define the following operations.

\begin{itemize}
    \item \textbf{Slope scaling ($c \cdot \mathcal{P}$):} The polygon obtained by multiplying every slope of $\mathcal{P}$ by $c$. The horizontal lengths remain unchanged. Its segment multiset is $\{ (c\mu_i, h_i) \}_{i \in I}$.
    \item \textbf{Slope shifting ($\mathcal{P} + c$):} The polygon obtained by adding $c$ to every slope of $\mathcal{P}$. Its segment multiset is $\{ (\mu_i + c, h_i) \}_{i \in I}$.
    \item \textbf{Polygon addition ($\mathcal{P}_1 \oplus \mathcal{P}_2$):} The polygon obtained by taking the multiset union of the segments of $\mathcal{P}_1$ and $\mathcal{P}_2$. If $f_1(x)$ and $f_2(x)$ are entire series with Newton polygons $\mathcal{P}_1$ and $\mathcal{P}_2$, then $\mathcal{P}_1 \oplus \mathcal{P}_2$ is the Newton polygon of their product $f_1(x) f_2(x)$.
    \item \textbf{Length scaling ($\mathcal{P}^{\oplus r}$):} For any real number $r > 0$ this denotes the polygon obtained by multiplying the horizontal length of every segment in $\mathcal{P}$ by $r$. Its segment multiset is $\{ (\mu_i, r h_i) \}$. 
    \item \textbf{Lower bounds ($\mathcal{P}_1 \ge \mathcal{P}_2$):} We say $\mathcal{P}_1 \ge \mathcal{P}_2$ if the polygon in $\R^2$ of $\mathcal{P}_1$ lies on or above the polygon $\mathcal{P}_2$ for all $x \ge 0$. If $\mathcal{P}_1$ has finite horizontal length $h$, we possibly extend its graph to match the horizontal length of $\mathcal{P}_2$ by setting the $y$-coordinates equal to $+\infty$ for $x > h$.
\end{itemize}

Notice that for an integer $k \ge 1$, length scaling coincides with repeated polygon addition: $\mathcal{P}^{\oplus k} = \mathcal{P} \oplus \dots \oplus \mathcal{P}$ ($k$ times). A fractional exponent such as $\mathcal{P}^{\oplus 1/2}$ represents a compression, in this case halving the horizontal length of every segment (slopes remain unchanged).

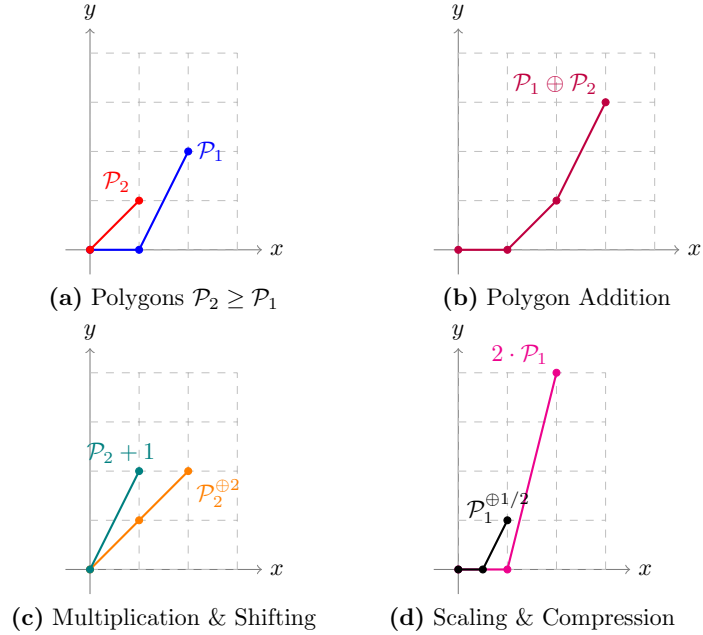
\begin{figure}[H]\label{fig: polygons}
    \centering
    \begin{tikzpicture}[scale=0.65, every node/.style={scale=0.9}]
        
        \begin{scope}[xshift=0cm, yshift=0cm]
            \draw[->, gray] (-0.5, 0) -- (3.5, 0) node[right, black] {$x$};
            \draw[->, gray] (0, -0.5) -- (0, 4.5) node[above, black] {$y$};
            \draw[step=1cm, lightgray, very thin, dashed] (0,0) grid (3,4);
            
            \draw[thick, blue] (0,0) -- (1,0) -- (2,2) node[right] {$\mathcal{P}_1$};
            \filldraw[blue] (0,0) circle (2pt) (1,0) circle (2pt) (2,2) circle (2pt);
            
            \draw[thick, red] (0,0) -- (1,1) node[above left] {$\mathcal{P}_2$};
            \filldraw[red] (0,0) circle (2pt) (1,1) circle (2pt);
            
            \node at (1.5, -1) {\textbf{(a)} Polygons $\mathcal{P}_2 \geq \mathcal{P}_1$};
        \end{scope}

        \begin{scope}[xshift=7.5cm, yshift=0cm]
            \draw[->, gray] (-0.5, 0) -- (4.5, 0) node[right, black] {$x$};
            \draw[->, gray] (0, -0.5) -- (0, 4.5) node[above, black] {$y$};
            \draw[step=1cm, lightgray, very thin, dashed] (0,0) grid (4,4);
            
            \draw[thick, purple] (0,0) -- (1,0) -- (2,1) -- (3,3);
            \filldraw[purple] (0,0) circle (2pt) (1,0) circle (2pt) (2,1) circle (2pt) (3,3) circle (2pt);
            
            \node[above left, purple] at (3,3) {$\mathcal{P}_1 \oplus \mathcal{P}_2$};
            \node at (2, -1) {\textbf{(b)} Polygon Addition};
        \end{scope}

        \begin{scope}[xshift=0cm, yshift=-6.5cm]
            \draw[->, gray] (-0.5, 0) -- (3.5, 0) node[right, black] {$x$};
            \draw[->, gray] (0, -0.5) -- (0, 4.5) node[above, black] {$y$};
            \draw[step=1cm, lightgray, very thin, dashed] (0,0) grid (3,4);
            
            \draw[thick, orange] (0,0) -- (2,2);
            \filldraw[orange] (0,0) circle (2pt) (1,1) circle (2pt) (2,2) circle (2pt);
            \node[below right, orange] at (2,2) {$\mathcal{P}_2^{\oplus 2}$};
            
            \draw[thick, teal] (0,0) -- (1,2);
            \filldraw[teal] (0,0) circle (2pt) (1,2) circle (2pt);
            \node[above left, teal] at (1.5,2) {$\mathcal{P}_2 + 1$};
            
            \node at (1.5, -1) {\textbf{(c)} Multiplication \& Shifting};
        \end{scope}
        
        \begin{scope}[xshift=7.5cm, yshift=-6.5cm]
            \draw[->, gray] (-0.5, 0) -- (3.5, 0) node[right, black] {$x$};
            \draw[->, gray] (0, -0.5) -- (0, 4.5) node[above, black] {$y$};
            \draw[step=1cm, lightgray, very thin, dashed] (0,0) grid (3,4);
            
            \draw[thick, magenta] (0,0) -- (1,0) -- (2,4);
            \filldraw[magenta] (0,0) circle (2pt) (1,0) circle (2pt) (2,4) circle (2pt);
            \node[above left, magenta] at (2,4) {$2 \cdot \mathcal{P}_1$};
            
            \draw[thick, black] (0,0) -- (0.5,0) -- (1,1);
            \filldraw[black] (0,0) circle (2pt) (0.5,0) circle (2pt) (1,1) circle (2pt);
            \node[below right, black] at (0,1.75) {$\mathcal{P}_1^{\oplus 1/2}$};
            
            \node at (1.5, -1) {\textbf{(d)} Scaling \& Compression};
        \end{scope}

    \end{tikzpicture}
    \caption{\textbf{(a)} $\mathcal{P}_1$ has segment set $\{(0, 1), (2, 1)\}$ and $\mathcal{P}_2$ has segment $\{(1,1)\}$. \textbf{(b)} $\mathcal{P}_1 \oplus \mathcal{P}_2$ has segments $\{(0,1), (1, 1), (2, 1)\}$. \textbf{(c)} $\mathcal{P}_2^{\oplus 2}$ has segments $\{(1, 1), (1, 1)\}$, and $\mathcal{P}_2 + 1$ has segment $\{(2, 1)\}$ \textbf{(d)} $2 \cdot \mathcal{P}_1$ has segments $\{(0, 1), (4, 1)\}$, and $\mathcal{P}_1^{\oplus 1/2}$ has segments $\{(0, 1/2), (2, 1/2)\}$.}
    \label{fig:polygon_arithmetic}
\end{figure}

\section{Newton-over-Hodge for the $C$-function}\label{sec:noh}

In this section we give a lower bound for the $T$-adic Newton polygons of the $E_m$ functions. Recall from Section \ref{sec:meromorphy} the space $B_\Omega$. The shift operator $\sigma$ acts on $B_\Omega$, and since $\sigma^d = 1$, we may define the eigenspace ($\zeta \in \mu_d$)
\[
B_\Omega^{(\zeta)} := \ker(\sigma - \zeta \mid B_\Omega).
\]
Define $\rho_\zeta := \frac1d \sum_{j=0}^{d-1} \zeta^{-j} \sigma^j$. 

\begin{lemma}
$\rho_\zeta: B_\Omega \rightarrow B_\Omega^{(\zeta)}$ is a surjective projection.
\end{lemma}

\begin{proof}
Let $x \in B_\Omega$. Applying $\sigma$, we have
\[
\sigma(\rho_\zeta(x)) = \frac{1}{d} \sum_{j=0}^{d-1} \zeta^{-j} \sigma^{j+1}(x) = \zeta \!\left( \frac{1}{d} \sum_{j=0}^{d-1} \zeta^{-(j+1)} \sigma^{j+1}(x) \right) =  \zeta \rho_\zeta(x).
\]
Hence, $\rho_\zeta(x) \in B_\Omega^{(\zeta)}$. 

Next, let $y \in B_\Omega^{(\zeta)}$. By definition, $\sigma(y) = \zeta y$, and iterating this gives $\sigma^j(y) = \zeta^j y$. Thus
\[
\rho_\zeta(y) = \frac{1}{d} \sum_{j=0}^{d-1} \zeta^{-j} \sigma^j(y) = \frac{1}{d} \sum_{j=0}^{d-1} \zeta^{-j} \zeta^j y = \frac{1}{d} \sum_{j=0}^{d-1} y = y.
\]
\end{proof}

Fix $m \mid d$, and define the space
\[
B_\Omega^{(m)} := \bigoplus_{\mathrm{ord}(\zeta) = m} B_\Omega^{(\zeta)}.
\]
Define the projection map $\rho_m : B_\Omega \rightarrow B_\Omega^{(m)}$ by
\[
\rho_m := \sum_{\mathrm{ord}(\zeta) = m} \rho_\zeta = \frac1d \sum_{j=0}^{d-1} b_m(j)\, \sigma^j \in \frac{1}{d}\Z[\sigma],
\]
where 
\[
b_m(j) := \sum_{\mathrm{ord}(\zeta) = m} \zeta^{-j} \in \Z.
\]
Since $\sigma$ and $\alpha_a$ commute,  $\alpha_a$ is stable on $B_\Omega^{(m)}$. Define
\[
D_m(T, s) := \det(I - s \alpha_a \mid B_\Omega^{(m)}).
\]

\begin{lemma}\label{lem:Dm_power}
For each $m \mid d$, we have $D_m(T, s) = E_m(T, s)^{\varphi(m)}$ 
\end{lemma}

\begin{proof}
On $B_\Omega^{(\zeta)}$ the generalized $\lambda$-eigenspace of $\alpha_a$ is $V_\lambda \cap
B_\Omega^{(\zeta)} = V_\lambda^{(\zeta)}$, which has dimension $d_{\lambda, \zeta} = c_{m, \lambda}$. Thus
\[
\det(I - s \alpha_a \mid B_\Omega^{(\zeta)}) = \prod_{\lambda \neq 0}(1 - \lambda s)^{c_{m, \lambda}} = E_m(T, s).
\]
In particular, it is independent of $\zeta$. Thus,
\begin{align*}
D_m(T, s) &= \det( I - s\alpha_a \mid \bigoplus_{\mathrm{ord}(\zeta) = m} B_\Omega^{(\zeta)}) \\
&= \prod_{\mathrm{ord}(\zeta) = m} \det(I - s \alpha_a \mid B_\Omega^{(\zeta)}) \\
&= \prod_{\mathrm{ord}(\zeta) = m} E_m(T, s) \\
&= E_m(T, s)^{\varphi(m)}.
\end{align*}
\end{proof}

We now move to the Newton polygon. Due to a possible inversion of the prime $p$ from the factor $1/d$ in the projection map, we define the following. Set $R := \Q_q[[\pi^{1/D}]]$, and define $B_R := B \widehat{\otimes}_{\OT} R$. Since $\rho_m \in \frac{1}{d}\Z[\sigma]$, we may define $B_R^{(m)} := \rho_m B_R$. Recall, for $\xi = \sum_{u \in M} a_u \pi^{w(u)} y^u \in B_R$ we have $\ord_T \xi := \inf_{u \in M} \ord_T(a_u)$. 

\begin{lemma}\label{lem: ord of rho_m}
For every $\xi \in B_R$, $\ord_T(\rho_m(\xi)) \ge \ord_T(\xi)$.
\end{lemma}

\begin{proof}
Since $\ord_T \sigma^j(\xi) = \ord_T \xi$, and $\rho_m = \frac{1}{d} \sum_{j=0}^{d-1} b_m(j) \sigma^j$ with $b_m(j) \in \Z$, the result follows. 
\end{proof}

\begin{lemma}
$B_R^{(m)}$ has an orthonormal basis $\{\xi_i^{(m)}\}_{i \geq 1}$ over $R$ consisting of weight-homogeneous elements.
\end{lemma}

\begin{proof}
For a fixed weight $r \in \frac{1}{D}\Z_{\geq 0}$, let $L_r := \mathrm{span}_R \{ \pi^{w(u)} y^u \mid w(u) = r\}$. Since $\sigma$ is weight-preserving, so is $\rho_m$, and thus $\rho_m$ is an idempotent on $L_r$. Set $W_{r,R}^{(m)} := \rho_m L_r$. Then $W_{r,R}^{(m)}$ is a direct summand of $L_r$, and since $R$ is a discrete valuation ring, we see that $W_{r,R}^{(m)}$ is a finite free $R$-module.

We now construct an orthonormal basis of $W_{r,R}^{(m)}$. Since $R = \Q_q[[\pi^{1/D}]]$ is a discrete valuation ring with uniformizer $\pi^{1/D}$ and residue field $\Q_q$, we see that
\[
\overline{W}_{r,R}^{(m)} := W_{r,R}^{(m)} \Big/ \pi^{1/D} W_{r,R}^{(m)}
\]
is a finite-dimensional vector space over $\Q_q$. Let $\{ \bar{\xi}_{r, 1}, \dots, \bar{\xi}_{r, h} \}$ be a $\Q_q$-basis of $\overline{W}_{r,R}^{(m)}$. Lift these to any elements $\mathcal{B}_r := \{ \xi_{r, 1}, \dots, \xi_{r, h} \}$ in $W_{r,R}^{(m)}$. Then $\mathcal{B}_r$ is a basis of $W_{r,R}^{(m)}$ over $R$ consisting of weight-$r$ elements with the property: for $\xi = \sum a_i \xi_{r, i} \in W_{r,R}^{(m)}$ with $a_i \in R$, then $\ord_T(\xi) = \min_{1 \leq i \leq h} \ord_T a_i$. 

Let $\mathcal{B} := \bigcup_r \mathcal{B}_r$ be the union of these bases over $r \in \frac 1D\Z$. We now show that $\mathcal{B}$ is an orthonormal basis of $B_R^{(m)}$. Let $\xi \in B_R^{(m)}$, then $\xi = \rho_m(\zeta)$ for some $\zeta \in B_R$. Write $\zeta = \sum \zeta_r$ where $\zeta_r \in L_r$ and $\ord_T \zeta_r \to \infty$ as $r \to \infty$. Then $\xi = \sum \rho_m( \zeta_r)$ with $\rho_m(\zeta_r) \in W_{r,R}^{(m)}$. By Lemma \ref{lem: ord of rho_m}, $\ord_T \rho_m(\zeta_r) \ge \ord_T(\zeta_r)$, and so $\ord_T \rho_m(\zeta_r) \to \infty$ as $r \to \infty$. Since each $\rho_m(\zeta_r)$ lives in its own weight-$r$ space, their $T$-adic valuations cannot affect each other. Consequently, $\ord_T \xi = \inf_r \ord_T \rho_m(\zeta_{r})$.

Using the basis $\mathcal{B}_r$, write $\rho_m(\zeta_r) = \sum a_{r, i} \xi_{r, i}$ for some $a_{r, i} \in R$. By construction of the basis, $\ord_T \rho_m(\zeta_r) = \min_i \ord_T a_{r, i}$. Thus, $\ord_T \xi = \inf_{r, i} \ord_T a_{r, i}$. This shows $\mathcal{B}$ is an orthonormal basis of $B_R^{(m)}$. 
\end{proof}

Since $B_\Omega^{(m)} = B_R^{(m)} \widehat{\otimes}_R \Omega$, we may compute $D_m(T, s)$ using $R$ instead of $\Omega$: 
\[
D_m(T, s) = \det(I - s \alpha_a \mid B_\Omega^{(m)}) = \det(I - s \alpha_a \mid B_R^{(m)}).
\]
Let $\alpha := \tau^{-1} \circ \psi_p \circ F(y)$ so that $\alpha_a = \alpha^a$. On the orthonormal basis $\{\xi_i^{(m)}\}_{i \geq 1}$ of $B_R^{(m)}$, write
\[
\alpha( \xi_i^{(m)} ) = \sum_{j =1}^\infty \tilde{A}_{ij}^{(m)} \xi_j^{(m)}
\]
with $\tilde{A}_{ij}^{(m)} \in R$. 

\begin{lemma}\label{lem:adapted_basis}
For every $i, j \geq 1$, we have $\ord_T(\tilde{A}_{ij}^{(m)}) \;\ge\; (p-1) w(\xi_j^{(m)})$.
\end{lemma}

\begin{proof}
Recall from the proof of Theorem~\ref{thm:complete_continuity} that
\[
\alpha(\pi^{w(u)} y^u) = \sum_{v \in M} A_{u,v} \pi^{w(v)} y^v,
\]
where the coefficients satisfy $\ord_T(A_{u,v}) \ge (p-1)w(v)$. 

We abuse notation slightly in the following; if $\xi$ is a weight-homogeneous element, we write $w(\xi)$ for its weight. With this in mind, since the elements in the basis $\{\xi_i^{(m)}\}$ of $B_R^{(m)}$ are weight-homogeneous, each $\xi_i^{(m)}$ of weight $w(\xi_i^{(m)})$ can be uniquely expressed as a finite linear combination of the standard basis elements of the same weight:
\[
\xi_i^{(m)} = \sum_{w(u)=w(\xi_i^{(m)})} b_u \pi^{w(u)} y^u \qquad (b_u \in R).
\]
Applying $\alpha$ gives:
\begin{align*}
\alpha(\xi_i^{(m)}) &=  \sum_{w(u)=w(\xi_i^{(m)})} \tau^{-1}(b_u) \alpha(\pi^{w(u)} y^u) \\
&= \sum_{w(u) = w(\xi_i^{(m)})} \tau^{-1}(b_u) \sum_{v \in M} A_{u,v} \pi^{w(v)} y^v \\
&= \sum_{v \in M} C_v \pi^{w(v)} y^v,
\end{align*}
where $C_v := \sum_{w(u)=w(\xi_i^{(m)})} \tau^{-1}(b_u) A_{u,v}$. Since $\tau^{-1}$ acts on Teichm\"uller units and $\ord_T(b_u) \geq 0$, we have
\[
\ord_T(C_v) \ge \min_u \big( \ord_T(b_u) + \ord_T(A_{u,v}) \big) \ge (p-1)w(v).
\]
Denote by $\eta_r := \sum_{w(v)=r} C_v \pi^{w(v)} y^v$ the weight-$r$ component of $\alpha(\xi_i^{(m)})$. Note that
\[
\ord_T(\eta_r) = \inf_{w(v)=r} \ord_T(C_v) \;\ge\; (p-1)r.
\]

Now, since $\eta_r \in W_{r, R}^{(m)}$, we may write $\eta_r$ in terms of the orthonormal basis of $W_{r, R}^{(m)}$:
\[
\eta_r = \sum_{w(\xi_j^{(m)})=r} \tilde{A}_{ij}^{(m)} \xi_j^{(m)}
\]
for some $\tilde{A}_{ij}^{(m)} \in R$. By construction, we have $\ord_T(\eta_r) = \min_{w(\xi_j^{(m)})=r} \ord_T(\tilde{A}_{ij}^{(m)})$. Hence,
\[
\ord_T(\tilde{A}_{ij}^{(m)}) \;\ge\; \min_{w(\xi_j^{(m)})=r} \ord_T(\tilde{A}_{ij}^{(m)}) = \ord_T(\eta_r) \;\ge\; (p-1)r = (p-1)w(\xi_j^{(m)}).
\]
This completes the proof.
\end{proof}

For each $m \mid d$, we define a Hodge polygon $\HP_m(\Delta)$ as follows. Set $h_m(r) := \mathrm{rank}_R W_{r, R}^{(m)}$. Using the notation from Section \ref{sec: polygons}, define $\HP_m(\Delta)$ as the segment multiset (segments are listed as (slope, horz. length))
\[
\{ (r, h_m(r)) : r \in \tfrac1D \Z_{\ge 0} \,\}.
\]
We note that $\HP_m(\Delta)$ depends only on $\bd$, $\Delta$, and a divisor $m$ of $d$. Denote by $\NP_T\!\big( E_m \big)$ the $T$-adic Newton polygon of $E_m(T, s)$.

\begin{theorem}\label{thm:native_orbit_noh}
Using notation from Section \ref{sec: polygons},  for every $m \mid d$,
\[
\NP_T\!\big( E_m \big) \;\ge\; a(p-1) \cdot \HP_m(\Delta)^{\oplus \frac{1}{\varphi(m)}}.
\]
\end{theorem}

\begin{proof}
By Lemma \ref{lem:adapted_basis}, the matrix of $\alpha$ on $B_R^{(m)}$ satisfies the lower bound $\ord_T(\tilde{A}_{ij}^{(m)}) \ge (p-1)w(\xi_j^{(m)})$. Since $\alpha_a = \alpha^a$, by a standard argument \cite[Section 7]{Dwork-zetafunctionof-1964}, the $T$-adic Newton polygon of $\alpha_a$ satisfies 
\[
\NP_T(D_m) = \NP_T\!\big( \det(I - s \alpha_a \mid B_R^{(m)}) \big) \ge a(p-1) \HP_m(\Delta).
\]
Now, by Lemma~\ref{lem:Dm_power}, $D_m = E_m^{\varphi(m)}$. Thus the reciprocal zeros of $D_m$ are the same as $E_m$ but repeated $\varphi(m)$ times. In terms of Newton polygons, this means  $\NP_T(D_m)$ is $\NP_T(E_m)$ with every horizontal length multiplied by $\varphi(m)$. Dividing horizontal lengths by $\varphi(m)$ gives 
\[
\NP_T(E_m) = \NP_T(D_m)^{\oplus \frac{1}{\varphi(m)}} \ge a(p-1) \cdot \HP_m(\Delta)^{\oplus \frac{1}{\varphi(m)}}.
\]
\end{proof}

Define the polygons:
\[
\HP^+(\Delta) := \bigoplus_{\substack{m \in \mathcal S_d \\ \mu(m) = 1}} \HP_m(\Delta)^{\oplus \frac{1}{\varphi(m)}},
\qquad
\HP^-(\Delta) := \bigoplus_{\substack{m \in \mathcal S_d \\ \mu(m) = -1}} \HP_m(\Delta)^{\oplus \frac{1}{\varphi(m)}}.
\]
Recall from Theorem \ref{thm:Cd_meromorphic} that we may write the $C$-function in reduced form as
\[
C_\bd(f, T; s) = \frac{A(T, s)}{B(T, s)},
\]
where $A(T, s),\, B(T, s) \in 1 + s\,\Z_p[[T]][[s]]$ are $T$-adic entire and coprime.

\begin{theorem}\label{thm:noh_entire_constituents}
The reduced numerator and denominator of the $C$-function satisfy
\[
\NP_T(A) \;\ge\; a(p-1) \HP^+(\Delta), \qquad \NP_T(B) \;\ge\; a(p-1) \HP^-(\Delta).
\]
\end{theorem}

\begin{proof}
Set $\mathcal A := \prod_{\mu(m) = 1} E_m$ and $\mathcal B := \prod_{\mu(m) = -1} E_m$. Then $C_\bd = \mathcal A / \mathcal B$ by Theorem~\ref{thm:structural_factorization}. Since, Newton polygons add
under multiplication of entire series, we have from Theorem~\ref{thm:native_orbit_noh}:
\[
\NP_T(\mathcal A) = \bigoplus_{\mu(m) = 1} \NP_T(E_m)
\;\ge\; \bigoplus_{\mu(m) = 1} \Big( a(p-1) \cdot \HP_m(\Delta)^{\oplus \frac{1}{\varphi(m)}} \Big) 
= a(p-1) \cdot \HP^+(\Delta),
\]
and likewise $\NP_T(\mathcal B) \ge a(p-1) \HP^-(\Delta)$.

To pass to the reduced factors $A, B$, which may differ from $\mathcal A, \mathcal B$ since the $E_m$ need not be pairwise coprime, we need only note that the Newton polygon rises when a root is removed. Hence, 
\[
\NP_T(A) \ge \NP_T(\mathcal A) \ge a(p-1) \HP^+(\Delta)
\]
and similarly $\NP_T(B) \ge a(p-1) \HP^-(\Delta)$.
\end{proof}

\section{Specialization of Newton-over-Hodge}\label{subsec:noh_specialization}

Let $\psi$ be a continuous, locally constant, additive character of $\Z_p$ of conductor $p^m$. Set $\pi_\psi := \psi(1) - 1$, and recall from Section~\ref{sec:rationality} that we have the specialization $C_{\bd, \psi}(f; s) := C_\bd(f, T; s)\big|_{T = \pi_\psi}$. Let $\HP^\pm(\Delta)$ be the polygons from the previous section.

\begin{lemma}\label{lem:np_transport}
Let $H = 1 + \sum_{k \ge 1} a_k(T)\, s^k \in 1 + s\,\Z_p[[T]][[s]]$ be $T$-adic entire. Then
$H_\psi := H\big|_{T = \pi_\psi} \in 1 + s\,\Z_p[\zeta_{p^m}][[s]]$ is $\pi_\psi$-adic entire,
and
\[
\NP_{\pi_\psi}(H_\psi) \;\ge\; \NP_T(H).
\]
\end{lemma}

\begin{proof}
By Lemma~\ref{lem:specialization} the substitution $T \mapsto \pi_\psi$ is a continuous homomorphism $\Z_p[[T]] \to \Z_p[\zeta_{p^m}]$ with $\ord_{\pi_\psi} a_k(\pi_\psi) \ge \ord_T a_k(T)$ for every $k$. Hence, the $\pi_\psi$-adic Newton polygon of $H_\psi$ lies on or above the $T$-adic Newton polygon of $H$. The result follows. 
\end{proof}

\begin{theorem}\label{thm:specialized_noh}
Write $C_{\bd, \psi}(f; s) = A'(s) / B'(s)$ with $A', B' \in 1 + s \Z_p[\zeta_{p^m}][[s]]$ coprime. Then
\[
\NP_{\pi_\psi}(A') \;\ge\; a(p-1) \HP^+(\Delta), \qquad
\NP_{\pi_\psi}(B') \;\ge\; a(p-1) \HP^-(\Delta).
\]
\end{theorem}

\begin{proof}
At the end of the previous section, we wrote $C_\bd(f, T; s) = A(T, s) / B(T, s)$. Let $A_\psi$ and $B_\psi$ be the specialization of $A(T, s)$ and $B(T, s)$ at $T = \pi_\psi$. By Theorem~\ref{thm:noh_entire_constituents} and Lemma~\ref{lem:np_transport}, we have
\[
\NP_{\pi_\psi}(A_\psi) \ge \NP_T(A) \ge a(p-1)\HP^+(\Delta),
\]
\[
\NP_{\pi_\psi}(B_\psi) \ge \NP_T(B) \ge a(p-1)\HP^-(\Delta).
\]
The result follows since $A'$ and $B'$ are factors of $A_\psi$ and $B_\psi$, respectively.
\end{proof}

\section{Newton-over-Hodge for the $L$-function}\label{sec:noh_L_func}

In this section, we obtain Newton-over-Hodge bounds for both the partial $T$-adic $L$-function $L_\bd(f, T; s)$ and its specialization $L_{\bd, \psi}(f; s)$. 

Recall from Theorem~\ref{thm:L_C_relation} that we have the relation
\[
L_\bd(f, T; s)^{(-1)^{n-1}} = C_\bd(f, T; s)^{\delta_\bd}.
\]
Using the definition of the $\delta_\bd$ operator, this can be expanded as a product over all subsets $J \subseteq \{1, \dots, n\}$. Setting $d_J := \sum_{j \in J} d_j$, we have
\[
L_\bd(f, T; s)^{(-1)^{n-1}} = \prod_{J \subseteq \{1, \dots, n\}} C_\bd\!\left(f, T; q^{d_J} s\right)^{(-1)^{|J|}}.
\]
Let $J_{\mathrm{even}}$ be the collection of subsets $J$ for which the cardinality $|J|$ is even, and $J_{\mathrm{odd}}$ be the collection of odd cardinality subsets. Let $C_\bd(f, T; s) = A(T, s) / B(T, s)$ be the reduced fraction from Theorem \ref{thm:Cd_meromorphic}. Then 
\begin{equation}\label{equ: L into Js}
L_\bd(f, T; s)^{(-1)^{n-1}} =  \frac{ \displaystyle\prod_{J \in J_{\mathrm{even}}} A(T, q^{d_J} s) \prod_{J \in J_{\mathrm{odd}}} B(T, q^{d_J} s) }{ \displaystyle\prod_{J \in J_{\mathrm{odd}}} A(T, q^{d_J} s) \prod_{J \in J_{\mathrm{even}}} B(T, q^{d_J} s) }.
\end{equation}
Write $L_\bd(f, T; s)^{(-1)^{n-1}} = \mathcal{A}(T, s) / \mathcal{B}(T, s)$ as a reduced fraction of coprime $T$-adic entire series. Then $\mathcal{A}$ is a factor of the numerator of \eqref{equ: L into Js}, and $\mathcal{B}$ a factor of the denominator. Set $\HP_\bd(\Delta) := \HP^+(\Delta) \oplus \HP^-(\Delta)$. 

\begin{theorem}
The $T$-adic Newton polygons of $\mathcal{A}$ and $\mathcal{B}$ are both bounded below by 
\[
a(p-1) \cdot \HP_\bd(\Delta)^{\oplus 2^{n-1}}.
\]
\end{theorem}

\begin{proof}
Since $q$ does not affect the $T$-adic valuation, for any subset $J$, we have $\NP_T\!\big(A(T, q^{d_J} s)\big) = \NP_T(A(T, s))$ and $\NP_T\!\big(B(T, q^{d_J} s)\big) = \NP_T(B(T, s))$. Applying Theorem~\ref{thm:noh_entire_constituents} to the numerator and denominator of \eqref{equ: L into Js}, we get
\begin{align*}
\NP_T(\mathcal{A}) &\;\ge\; \bigoplus_{J \in J_{\mathrm{even}}} a(p-1)\HP^+(\Delta) \;\oplus \bigoplus_{J \in J_{\mathrm{odd}}} a(p-1)\HP^-(\Delta), \\
\NP_T(\mathcal{B}) &\;\ge\; \bigoplus_{J \in J_{\mathrm{odd}}} a(p-1)\HP^+(\Delta) \;\oplus \bigoplus_{J \in J_{\mathrm{even}}} a(p-1)\HP^-(\Delta).
\end{align*}
The results follow by polygon arithmetic and that $\HP^\pm(\Delta)$ are independent of $J$.
\end{proof}

We now specialize to a character $\psi$. For $J \subseteq \{1, \ldots, n\}$, let $\varepsilon(J) \in \{+, -\}$ denote the parity of $|J|$, with $+$ if $|J|$ is even, and $-$ if $|J|$ is odd. Define the specialized Hodge polygons for the numerator and denominator as:
\begin{align*}
\HP_{\bd, \psi}^+(\Delta) &:= \bigoplus_{J \subseteq \{1, \dots, n\}} \Big( \frac{1}{p^{m-1}} \cdot \HP^{\varepsilon(J)}(\Delta) + d_J \Big), \\
\HP_{\bd, \psi}^-(\Delta) &:= \bigoplus_{J \subseteq \{1, \dots, n\}} \Big( \frac{1}{p^{m-1}} \cdot \HP^{-\varepsilon(J)}(\Delta) + d_J \Big).
\end{align*}

We note that since $\psi(1) = 1 + \pi_\psi$ is a primitive $p^m$-th root of unity, we have $\ord_{\pi_\psi}(p) = \varphi(p^m) = p^{m-1}(p-1)$. Consequently, $\ord_{\pi_\psi}(q) = a \varphi(p^m)$.

\begin{theorem}
Write $L_{\bd, \psi}(f; s)^{(-1)^{n-1}} = A_{\psi}(s) / B_{\psi}(s)$ as a reduced fraction with $A_\psi, B_\psi \in 1 + s \Z[\zeta_{p^m}][s]$. Then their $q$-adic Newton polygons satisfy
\[
\NP_q(A_{\psi}) \;\ge\; \HP_{\bd, \psi}^+(\Delta) \qquad \text{and} \qquad \NP_q(B_{\psi}) \;\ge\; \HP_{\bd, \psi}^-(\Delta).
\]
\end{theorem}

\begin{proof}
Consider a $T$-adic entire series $H(T, s) = \sum_{k=0}^\infty c_k(T) s^k$. Specializing $T = \pi_\psi$ and setting $s \mapsto q^r s$ gives $H(\pi_\psi, q^r s) = \sum_{k=0}^\infty c_k(\pi_\psi) q^{rk} s^k$. By Lemma \ref{lem:specialization}, $\ord_{\pi_\psi} c_k(\pi_\psi)  \ge \ord_T c_k(T)$. Since $\ord_{\pi_\psi}(q) = a \varphi(p^m)$, we have
\[
\ord_{\pi_\psi}\!\big(c_k(\pi_\psi) q^{rk}\big) \;\ge\; \ord_T(c_k(T)) + k r a \varphi(p^m).
\]
In terms of polygon arithmetic, this means
\[
\NP_{\pi_\psi}\!\big(H(\pi_\psi, q^r s)\big) \;\ge\; \NP_T(H) + r a \varphi(p^m).
\]

We apply this argument to each factor in the unreduced numerator and denominator of \eqref{equ: L into Js} using Theorem~\ref{thm:noh_entire_constituents}. The result follows by polygon arithmetic and since the reduced factors \(A_\psi(s)\) and \(B_\psi(s)\) are obtained through cancelling common roots, which possibly increases the Newton polygons. 
\end{proof}

\bibliographystyle{amsplain}
\bibliography{../References/References}

\end{document}